\documentclass[12pt,a4paper]{article}
\usepackage[a4paper]{geometry}

\usepackage[all]{xy}

\usepackage{amssymb, amsmath, amsthm}

\usepackage{mathptmx}
\usepackage[scaled=.90]{helvet}
\usepackage{courier}

\usepackage[pdftex]{hyperref}

\hypersetup{
  colorlinks = true,
  linkcolor = black,
  urlcolor = blue,
  citecolor = black,
}

\renewcommand{\phi}{\varphi}

\newcommand{\bbF}{\mathbb{F}}

\newcommand{\bbQ}{\mathbb{Q}}

\newcommand{\bbZ}{\mathbb{Z}}

\newcommand{\rmB}{\mathrm{B}}
\newcommand{\rmE}{\mathrm{E}}
\newcommand{\rmF}{\mathrm{F}}

\newcommand{\rmH}{\mathrm{H}}

\newcommand{\rmN}{\mathrm{N}}
\newcommand{\rmP}{\mathrm{P}}

\newcommand{\rmS}{\mathrm{S}}

\newcommand{\op}{\mathrm{op}}

\DeclareMathOperator{\B}{B\!}

\DeclareMathOperator{\Bhaut}{Bhaut}
\DeclareMathOperator{\Aut}{Aut}
\DeclareMathOperator{\BAut}{BAut}
\DeclareMathOperator{\GL}{GL}

\DeclareMathOperator{\Ad}{Ad}
\DeclareMathOperator{\Hom}{Hom}
\DeclareMathOperator{\Ext}{Ext}
\DeclareMathOperator{\Ker}{Ker}

\DeclareMathOperator{\rk}{rk}
\DeclareMathOperator{\gr}{gr}
\DeclareMathOperator{\Lie}{Lie}

\DeclareMathOperator*{\colim}{colim}

\newtheorem{theorem}{Theorem}
\newtheorem{proposition}[theorem]{Proposition}

\theoremstyle{definition}
\newtheorem{definition}[theorem]{Definition}
\newtheorem{example}[theorem]{Example}

\newtheorem{remark}[theorem]{Remark}

\numberwithin{theorem}{section}
\numberwithin{equation}{section}

\title{Twisted homological stability for extensions and automorphism groups of free nilpotent groups}

\author{Markus Szymik}

\date{January 2014}

\begin{document}

\maketitle

\renewcommand{\abstractname}{}

\begin{abstract}
\noindent 
We prove twisted homological stability with polynomial coefficients for automorphism groups of free nilpotent groups of any given class. These groups interpolate between two extremes for which homological stability was known before, the general linear groups over the integers and the automorphism groups of free groups. The proof presented here uses a general result that applies to arbitrary extensions of groups, and that has other applications as well.\\
\phantom{M}
\end{abstract}

%%%

\section*{Introduction}

Various notions of stability have influenced the prospering and thriving of algebraic~K-theory from its classical roots in algebra to its higher branches with connections to topology. See~\cite{Bass},~\cite{Lam+Siu} and~\cite{van_der_Kallen}, for example. For instance, the stable homology of the general linear groups~$\GL_r(\bbZ)$ over the ring~$\bbZ$ of integers, when~\hbox{$r\to\infty$}, can be interpreted as the homology of the algebraic~K-theory space of~$\bbZ$. Quillen, one of the founders of the theory, was also the first to prove homological stability for general linear groups~(over certain fields). Borel proved results that implied integral statements, but only the~(independent) work of Charney~\cite{Charney}, and Maazen (unpublished) and van der Kallen~\cite{van_der_Kallen} completed the picture for the general linear groups over the ring~$\bbZ$ of integers. 

%%%

The general linear groups~$\GL_r(\bbZ)=\Aut(\bbZ^r)$ are the automorphism groups of the free abelian groups of rank~$r$. A related family of groups is given by the automorphism groups~$\Aut(\rmF_r)$ of the free groups~$\rmF_r$ on~$r$ generators, without any commutativity assumption. Corresponding homological stability results for these groups are more recent, see~\cite{Hatcher} and~\cite{Hatcher+Vogtmann}. 

%%%

An interpolation between these two sequences of groups, namely~$\GL_r(\bbZ)$ and~$\Aut(\rmF_r)$, is given by relaxing the commutativity condition~(that is satisfied only by abelian groups) to higher nilpotency classes~$c\geqslant1$. The free nilpotent groups~$\rmN^c_r$ of class~$c$ give rise to towers
\[
\rmF_r\longrightarrow\cdots\longrightarrow\rmN^3_r\longrightarrow\rmN^2_r\longrightarrow\rmN^1_r=\bbZ^r
\]
of groups, and also (albeit slightly less obviously) to towers
\[
\Aut(\rmF_r)\longrightarrow\cdots\longrightarrow\Aut(\rmN^3_r)\longrightarrow\Aut(\rmN^2_r)\longrightarrow\Aut(\rmN^1_r)=\GL_r(\bbZ)
\]
of their automorphism groups.

%%%

In this paper, we prove homological stability, when~\hbox{$r\to\infty$}, for the automorphism groups~$\Aut(\rmN^c_r)$ of free nilpotent groups of any fixed class~$c$ with twisted coefficients~(Theorem~\ref{thm:HS_nil}). One may argue whether or not the main coefficients of interest are the constant ones, but the proof that we present here is inductive, and the more general statement turns out to keep the induction on fire. Each step of the induction uses a general result~(Theorem~\ref{thm:inflation}) that applies to arbitrary extensions of groups, and that has other applications as well. For example, in the final Section~\ref{sec:other}, we explain how to deduce the (known) homological stability results for wreath products and braid groups from the corresponding property of the symmetric groups.

%%%

\section{Inflation of homological stability}

In this section, after reviewing a general setup for homological stability questions, we extend this framework so as to include morphisms that allow us to compare different families of groups. After these preliminaries, we state and prove our main general result, Theorem~\ref{thm:inflation}.

\subsection{Homological stability}

Let us begin by introducing the terminology that we will use throughout the paper.

\begin{definition}
A sequence
\[
G_\bullet=(
G_0\longrightarrow
G_1\longrightarrow
G_2\longrightarrow\cdots)
\]
of group homomorphism that can be composed as displayed will be called, by abuse of language, a {\em sequence of groups}.
\end{definition}

\begin{example}
The sequence~$\rmS_\bullet$ of the symmetric groups~$\rmS_r$ of permutations of sets~$\{1,\dots,r\}$. The group homomorphisms~$\rmS_r\to \rmS_{r+1}$ are given by extending a given permutation so that~$r+1$ is a fixed point.
\end{example}

\begin{example}
The sequence~$\GL_\bullet(\bbZ)$ of the general linear groups~$\GL_r(\bbZ)$ of automorphisms of the free abelian group~$\bbZ^r$. The group homomorphisms~$\GL_r(\bbZ)\to \GL_{r+1}$ are given by extending a given automorphism so that the~$(r+1)$-st generator is fixed.
\end{example}

\begin{definition}
If~$G_\bullet$ is a sequence of groups, then a {\em~$G_\bullet$-module} is a sequence
\[
M_\bullet=(
M_0\longrightarrow
M_1\longrightarrow
M_2\longrightarrow\cdots)
\]
of homomorphisms that can be composed as displayed, together with additive~$G_r$-actions on each~$M_r$ such that~$M_r\to M_{r+1}$ is~$\GL_r(\bbZ)$-equivariant with respect to the action of~$\GL_r(\bbZ)$ on~$M_{r+1}$ via restriction along~$\GL_r(\bbZ)\to \GL_{r+1}$.
\end{definition}

\begin{example}
If~$A$ is an abelian group then the constant sequence~$(A=A=A=\dots)$ is a~$G_\bullet$-module for all sequences of groups, if~$A$ gets the trivial~$G_r$-action for each~$r$.
\end{example}

If~$M_\bullet$ is a~$G_\bullet$-module, then the group homology of the groups~$G_r$ with coefficients in the corresponding~$G_r$-modules~$M_r$ is related by homomorphisms
\begin{equation}\label{eq:homs}
\rmH_d(G_r;M_r)\longrightarrow\rmH_d(G_r;M_{r+1})\longrightarrow\rmH_d(G_{r+1};M_{r+1})
\end{equation}
of abelian groups.

\begin{definition}
A sequence~$G_\bullet$ is {\em homologically stable} with respect to a~$G_\bullet$-module~$M_\bullet$ if in all dimensions~$d$, the compositions \eqref{eq:homs} are eventually isomorphisms.
\end{definition}

We will use similar terminology if we have a class of~$G_\bullet$-modules instead of just one, for example the class of trivial~$G_\bullet$-modules. 

Regardless of the answer to the question if the sequence of homomorphisms~\eqref{eq:homs} stabilizes or not, we also have the following definition.

\begin{definition}
The colimit 
\[
\rmH_d(G_\bullet;M_\bullet)=\colim_r\rmH_d(G_r;M_r)
\]
is called the {\em stable homology} of~$G_\bullet$ with respect to~$M_\bullet$. 
\end{definition}

In the case of stability, a calculation of the stable homology immediately implies infinitely many unstable values.

\subsection{Morphisms}

The main aim of the present paper is to describe a means that allows to compare homological stability for different sequences of groups. For this purpose we now define the appropriate notion of morphisms between them, so that they form a category.

\begin{definition}
A {\em morphism}~$q_\bullet\colon G_\bullet\to Q_\bullet$ between two sequences of groups is a sequence~$q_r\colon G_r\to Q_r$ of group homomorphisms that commute with the homomorphisms in the sequence.
\end{definition}

For each morphism~$q_\bullet\colon G_\bullet\to Q_\bullet$ between sequences of groups, the kernels~$\Ker(q_r)$ form another sequence~$\Ker(q_\bullet)$ of groups, and this comes with a canonical morphism~\hbox{$k_\bullet\colon\Ker(q_\bullet)\to G_\bullet$}.

For each morphism~$q_\bullet\colon G_\bullet\to Q_\bullet$ between sequences of groups, and for each~$Q_\bullet$ module~$M_\bullet$, there is a restricted~$G_\bullet$-modules~$q_\bullet^*M_\bullet$. For example, if~$M_\bullet=A$ is constant, so is~$q_\bullet^*M_\bullet=A$.

\subsection{The inflation theorem}

We can now formulate and prove our main general result that allows us to relate twisted homological stability for sequences of groups that are connected by a homomorphism~\hbox{$q_\bullet\colon G_\bullet\to Q_\bullet$}.

\begin{theorem}\label{thm:inflation}
Let~$q_\bullet\colon G_\bullet\to Q_\bullet$ be a surjective morphism between sequences of groups with kernel~$k_\bullet\colon K_\bullet=\Ker(q_\bullet)\to G_\bullet$. The sequence~$G_\bullet$ satisfies twisted homological stability with respect to a~$G_\bullet$-module~$M_\bullet$ if the sequence~$Q_\bullet$ satisfies twisted homological stability with respect to the~$Q_\bullet$-modules~$\rmH_t(K_\bullet;k_\bullet^*M_\bullet)$ for each~\hbox{$t\geqslant0$}.
\end{theorem}

\begin{proof}
For each~$r\geqslant0$, the extension
\[
1\longrightarrow 
K_r\longrightarrow
G_r\longrightarrow
Q_r\longrightarrow1
\]
gives rise to a Lyndon-Hochschild-Serre spectral sequence
\[
\rmE^2_{s,t}(r)\cong\rmH_s(Q_r;\rmH_t(K_r;k_r^*M_r))\Longrightarrow\rmH_{s+t}(G_r;M_r).
\]
By assumption, for each fixed pair~$(s,t)$, the sequence
\[
\cdots\longrightarrow\rmE^2_{s,t}(r)\longrightarrow\rmE^2_{s,t}(r+1)\longrightarrow\cdots
\]
stabilizes for large~$r$. It follows from the strong convergence of the spectral sequence that the sequence
\[
\cdots\longrightarrow\rmE^\infty_{s,t}(r)\longrightarrow\rmE^\infty_{s,t}(r+1)\longrightarrow\cdots
\]
also stabilizes for large~$r$. But this is the associated graded of a finite filtration of the maps \eqref{eq:homs}
\[
\cdots\longrightarrow\rmH_{s+t}(G_r;M_r)\longrightarrow\rmH_{s+t}(G_{r+1};M_{r+1})\longrightarrow\cdots,
\]
so that these are eventually isomorphisms as well.
\end{proof}

%%%

\section{Free nilpotent groups}

In this section, we introduce some notation and present some basic results about the free nilpotent groups of a given class~$c\geqslant1$. Only in the next section will we turn towards their automorphisms. All of these results are certainly well-known, and our only claims to originality here are for the exposition, where the focus is on homological methods rather than those from Lie theory.

Let~$G$ be a (discrete) group. For integers~$n\geqslant1$, subgroups~$\Gamma_n(G)$ are defined inductively by~$\Gamma_1(G)=G$ and~$\Gamma_{n+1}(G)=[G,\Gamma_n(G)]$. We also set~$\Gamma_\infty(G)$ to be the intersection of all the~$\Gamma_n(G)$. This gives a  series
\[
G=\Gamma_1(G)\geqslant\Gamma_2(G)\geqslant\dots\geqslant\Gamma_\infty(G)
\]
of normal subgroups, the {\em descending/lower central series} of~$G$. The associated graded group
\[
\gr(G)=\bigoplus_{n=1}^\infty\gr^n(G)
\]
with 
\[
\gr^n(G)=\Gamma_n(G)/\Gamma_{n+1}(G)
\]
is abelian, and comes with the structure of a graded Lie algebra. The first graded piece~$\gr^1(G)$ is just the abelianization of~$G$. A group is {\em abelian} if~$\Gamma_2(G)$ is trivial.
A group is {\em nilpotent} if some~$\Gamma_n(G)$ is trivial. More precisely, a group~$G$ is nilpotent of class (at most)~$c$ (for some~$c\geqslant 1$) if~$\Gamma_{c+1}(G)$ is trivial. This means that abelian groups are precisely the groups which are nilpotent of class~$1$. A group is {\em residually nilpotent} if~$\Gamma_\infty(G)~$ is trivial.

Let~$\rmF_r$ denote the free group on a set of~$r$ generators. The case~$r=1$ is somewhat special, since~$\rmF_1\cong\bbZ$ is abelian, so that~$\Gamma_2(\rmF_1)$ is trivial. All the other ones are not nilpotent, but residually nilpotent \cite{Magnus}. The first graded piece~$\gr^1(\rmF_r)$ is the free abelian group on~$r$ generators. This gives an induced homomorphism
\begin{equation}\label{eq:Witt1}
\Lie_r\overset{\cong}{\longrightarrow}\gr(\rmF_r)
\end{equation}
of graded Lie algebras, where
\[
\Lie_r\cong\bigoplus_{n=1}^\infty\Lie_r^n
\]
is the free Lie algebra (over the ring~$\bbZ$) on~$r$ generators. (See~\cite[Chapter~IV]{Serre} for an exposition.) It turns out, as has already been indicated, that~\eqref{eq:Witt1} is an isomorphism, see~\cite{Witt}. In particular 
\begin{equation}\label{eq:Witt2}
\Gamma_n(\rmF_r)/\Gamma_{n+1}(\rmF_r)=\gr(\rmF_r)\cong\Lie_r^n
\end{equation}
is a free abelian group of rank
\[
\rk(\Lie_r^n)=\frac{1}{n}\sum_{d|n}\mu(d)r^{n/d}
\]
over~$\bbZ$, where~$\mu$ is the M\"obius function.

The universal examples of nilpotent groups of class~$c\geqslant1$ are the quotients
\[
\rmN_r^c=\rmF_r/\Gamma_{c+1}(\rmF_r).
\] 
As the two extreme cases, we obtain~$\rmN_r^1\cong\bbZ^r$ and~$\rmN_r^\infty\cong\rmF_r$. We record the following description of the low-dimensional homology of the groups~$\rmN_r^c$ under investigation.

\begin{proposition}\label{prop:homology12}
There are isomorphisms
\begin{align*}
\rmH_1(\rmN_r^c;\bbZ)&\cong\frac{\Gamma_{1}(\rmF_r)}{\Gamma_{2}(\rmF_r)}\cong\Lie_r^1\cong\bbZ^r\\
\rmH_2(\rmN_r^c;\bbZ)&\cong\frac{\Gamma_{c+1}(\rmF_r)}{\Gamma_{c+2}(\rmF_r)}\cong\Lie_r^{c+1}
\end{align*}
of homology groups. In particular, these groups are free abelian of finite rank.
\end{proposition}

\begin{proof}
The first statement follows immediately from the fact that any surjective homomorphism~\hbox{$G\to Q$} of groups induces an isomorphism between the first integral homology groups as soon as the kernel is contained in the commutator subgroup. Here, the kernel of~$\rmF_r\to\rmN_r^c$ is~$\Gamma_{c+1}(\rmF_r)$ which is contained in 
the commutator subgroup~$\Gamma_2(\rmF_r)$.

For the second statement, Hopf's Theorem \cite{Hopf} (or the Lyndon-Hochschild-Serre spectral sequence) says that
\[
\rmH_2(\rmF_r/R;\bbZ)\cong\frac{R\cap[\rmF_r,\rmF_r]}{[\rmF_r,R]}
\]
for each normal subgroup~$R$ of~$\rmF_r$.
In the particular case~$R=\Gamma_{c+1}(\rmF_r)$, we get
\[
\Gamma_{c+1}(\rmF_r)\cap\Gamma_2(\rmF_r)=\Gamma_{c+1}(\rmF_r)
\]
and
\[
[\rmF_r,\Gamma_{c+1}(\rmF_r)]=\Gamma_{c+2}(\rmF_r),
\]
so that
\[
\rmH_2(\rmF_r/\Gamma_{c+1}(\rmF_r);\bbZ)\cong\frac{\Gamma_{c+1}(\rmF_r)}{\Gamma_{c+2}(\rmF_r)}\cong\Lie_r^{c+1}
\]
as claimed.
\end{proof}

\begin{proposition}
For each fixed rank~$r$ there is an extension
\begin{equation}\label{eq:Nextension}
0\longrightarrow\rmH_2(\rmN_r^{c-1};\bbZ)\longrightarrow\rmN_r^c\longrightarrow\rmN_r^{c-1}\longrightarrow1
\end{equation}
of groups.
\end{proposition}

\begin{proof}
This follows from the definitions of the groups as quotients, the isomorphism theorem, and the calculations in~ the preceding Proposition~\ref{prop:homology12}.
\end{proof}

\begin{remark}
Therefore, each group~$\rmN_r^c$ can be described inductively as an extension of~$\rmN_r^{c-1}$ by its Schur multiplier. This extension is central; in fact, the kernel of~\hbox{$\rmN_r^c\to\rmN_r^{c-1}$} is equal to the center of~$\rmN_r^c$. However, we note that this is not a universal central extension in the usual sense, since the groups involved are not perfect. But still, the extension is classified by the tautological class in
\[
\rmH^2(\rmN_r^{c-1};\rmH_2(\rmN_r^{c-1};\bbZ))\cong\Hom(\rmH_2(\rmN_r^{c-1};\bbZ),\rmH_2(\rmN_r^{c-1};\bbZ))
\]
which is the identity. The isomorphism is given by the universal coefficient theorem, since~$\Ext(\rmH_1(\rmN_r^{c-1};\bbZ),\rmH_2(\rmN_r^{c-1};\bbZ))=0$.
\end{remark}

\begin{remark}\label{rem:BN_are_manifolds}
Since the classifying spaces for free abelian groups are tori, the extension~\eqref{eq:Nextension} shows inductively that the classifying spaces of the groups~$\rmN_r^c$ can be constructed as torus bundles over tori. In particular, they can be modeled as compact manifolds without boundary, as long as~$c$ is finite. In fact, their universal covers can taken to be nilpotent Lie groups in which the groups embed as co-compact lattices. This follows from Malcev's theory, but it is also described explicitly for this special case in~\cite[2.1]{Kassabov}.
\end{remark}

%%%

\section{Automorphisms of free nilpotent groups}

In the previous section, we have defined the free nilpotent groups~$\rmN_r^c$ of class~$c$ on~$r$ generators. We are interested in the groups
\[
\Aut(\rmN_r^c)
\]
of their automorphism and the classifying spaces thereof. In the case~$c=1$, these are the general linear groups~$\GL_r(\bbZ)$, and in the limiting case~$c=\infty$, these are the automorphism groups~$\Aut(\rmF_r)$ of free groups. 

The aim of this section is to explain the way in that the groups~$\Aut(\rmN_r^c)$ vary with~$c$ in between, when~$r$ is fixed. (See also~\cite{Morita} for an exposition with a slightly different focus.) We will also let~$r$ vary, but only later, in order to address the question of homological stability. 

Since the kernel of the canonical homomorphism~$\rmN_r^c\to\rmN_r^{c-1}$ is the center of~$\rmN_r^c$, every automorphism of~$\rmN_r^c$ preserves it and induces an automorphism of the quotient~$\rmN_r^{c-1}$. This gives a homomorphism
\begin{equation}\label{eq:induct}
\Aut(\rmN_r^c)\longrightarrow\Aut(\rmN_r^{c-1}).
\end{equation}
This will be the means that allows for an inductive approach towards these groups, starting with the case when~$\Aut(\rmN_r^1)=\GL_r(\bbZ)$ is the general linear group.

Let us first identify the kernel of~\eqref{eq:induct}, compare~\cite{Johnson}. An element~$\alpha$ in the kernel agrees with the identity on~$\rmN_r^c$ modulo elements in the kernel of~$\rmN_r^c\to\rmN_r^{c-1}$ which has been identified with the Schur multiplier~$\rmH_2(\rmN_r^{c-1};\bbZ)$. Thus, we can define a map~$\alpha^\flat\colon\rmN_r^c\to\rmH_2(\rmN_r^{c-1};\bbZ)$ via~$\alpha^\flat(n)=\alpha(n)n^{-1}$. It is readily checked that~$\alpha^\flat$ is a homomorphism of groups. Since the target is abelian, this induces a homomorphism
\[
\alpha^\flat\colon\rmH_1(\rmN_r^{c-1};\bbZ)\cong\rmH_1(\rmN_r^c;\bbZ)\to\rmH_2(\rmN_r^{c-1};\bbZ)
\]
such that the composition
\[
\xymatrix@1{
\rmN_r^c\ar[r]&
\rmH_1(\rmN_r^{c-1};\bbZ)\ar[r]^{\alpha^\flat}&
\rmH_2(\rmN_r^{c-1};\bbZ)\ar[r]&
\rmN_r^c.
}
\]
recovers~$\alpha$ in the sense that the relation~$\alpha(n)=\alpha^\flat(n)n$ holds. Conversely, each homomorphism~$\beta\colon\rmH_1(\rmN_r^{c-1};\bbZ)\to\rmH_2(\rmN_r^{c-1};\bbZ)$ defines a map~$\beta^\sharp\colon\rmN_r^c\to\rmN_r^c$ by multiplication of the identity from the left with the composition
\[
\xymatrix@1{
\rmN_r^c\ar[r]&
\rmH_1(\rmN_r^{c-1};\bbZ)\ar[r]^\beta&
\rmH_2(\rmN_r^{c-1};\bbZ)\ar[r]&
\rmN_r^c,
}
\]
so that the relation~$\beta^\sharp(n)=\beta(n)n$ holds. It is now readily verified that the map~$\beta^\sharp$ is an automorphism, and that the induced map
\begin{equation}\label{eq:kernel}
\Hom(\rmH_1(\rmN_r^{c-1};\bbZ),\rmH_2(\rmN_r^{c-1};\bbZ))\to\Aut(\rmN_r^c)
\end{equation}
is an injective homomorphism. The arguments above can summarized to say that the image of~\eqref{eq:kernel} is the kernel of~\eqref{eq:induct}. More is true: It turns out that both~\eqref{eq:induct} and~\eqref{eq:kernel} can be composed so that they together give a description of the group~$\Aut(\rmN_r^c)$ as an extension of~$\Aut(\rmN_r^{c-1})$ by the free abelian group~$\Hom(\rmH_1(\rmN_r^{c-1};\bbZ),\rmH_2(\rmN_r^{c-1};\bbZ))$.

\begin{proposition}
For each fixed rank~$r$ there is an extension
\[
0
\longrightarrow\Hom(\rmH_1(\rmN_r^{c-1};\bbZ),\rmH_2(\rmN_r^{c-1};\bbZ))
\longrightarrow\Aut(\rmN_r^c)
\longrightarrow\Aut(\rmN_r^{c-1})
\longrightarrow1
\]
of groups.
\end{proposition}

The only thing left to show is that the homomorphism~$\Aut(\rmN_r^c)\to\Aut(\rmN_r^{c-1})$ is surjective, and this is the content of~\cite[Theorem~2.1]{Andreadakis}. 

\begin{remark}\label{rem:action}
Any element of the quotient group~$\Aut(\rmN_r^{c-1})$ acts on the kernel group via restriction along~$\Aut(\rmN_r^{c-1})\to\Aut(\rmN_r^1)=\GL_r(\bbZ)$. In particular, if an element of~$\Aut(\rmN_r^{c-1})$ acts trivially on~$\rmH_1(\rmN_r^{c-1};\bbZ)\cong\bbZ^r$, then its action on~$\rmH_2(\rmN_r^{c-1};\bbZ)$ is also trivial.
\end{remark}

\begin{remark}
For every group there is an isomorphism between~$\Aut(G)$ and the group of based homotopy classes of base point preserving homotopy self-equivalences of the classifying space~$\B G$. In fact, there is a homotopy equivalence
\[
\BAut(G)\simeq\Bhaut(\B G,\star).
\]
As a consequence of Remark~\ref{rem:BN_are_manifolds}, we see that the classifying spaces~$\BAut(\rmN_r^c)$ are of the form~$\Bhaut(X,\star)$ for certain (acyclic!) manifolds~$X$ that are of the homotopy type of the classifying spaces~$\B\rmN_r^c$ of the free nilpotent groups of class~$c$ and rank~$r$.
\end{remark}

%%%

\section{Homological stability}

In this section, we prove homological stability for the automorphism groups of free nilpotent groups of any given class~$c\geqslant1$ with polynomial coefficients. The proof is by induction on the nilpotency class, starting with the known case~$c=1$ of the general linear groups. We recall the relevant results for this special case first.

%%%

\subsection{Twisted homological stability for the general linear groups}

Let us first review the twisted homological stability for the general linear groups as far as we need it here. See \cite{Dwyer} and \cite{van_der_Kallen}.

The sequence~$\GL_\bullet(\bbZ)$ of general linear groups has a standard module~$\bbZ^\bullet$ with the defining action of~$\GL_r(\bbZ)$ on~$\bbZ^r$. Let~$\overline{\bbZ}^\bullet$ denote the inverse transpose. 

\begin{definition}\label{def:polyGL}
We say that a~$\GL_\bullet(\bbZ)$-modules~$M_\bullet$ is {\em polynomial} if~$M_\bullet\cong F(\bbZ^\bullet,\overline{\bbZ}^\bullet)$ for some polynomial functor~$F$ from the category of pairs of abelian groups to abelian groups, compare \cite[Section~3]{Dwyer}.
\end{definition}

For example, the constant modules are clearly polynomial in this sense, and so is the adjoint representation/module~\hbox{$\Ad_\bullet(\bbZ)\cong\bbZ^\bullet\otimes\overline{\bbZ}^\bullet$}, which is the main example that appears in~\cite{Dwyer}. A similar example is given by the module~$\Hom(\bbZ^\bullet,\Lambda^2\overline{\bbZ}^\bullet)$. This will show up later again (in Proposition~\ref{prop:Hom_is_polynomial}). An immediate application of the main stability theorem in~\cite{Dwyer}, using Lemma~3.1. in {\it loc.cit.},~gives the following result.

\begin{theorem}\label{thm:HS_GL}{\upshape\bf(Dwyer)}
The sequence~$\GL_\bullet(\bbZ)$ of general linear groups satisfies homological stability with respect to polynomial~$\GL_\bullet(\bbZ)$-modules.
\end{theorem}

\subsection{Stability for the automorphism groups of free nilpotent groups}

We are now ready to prove twisted homological stability for the automorphism groups of free nilpotent groups with polynomial coefficients. We first define the latter.

\begin{definition}
We say that an~$\Aut(\rmN^c_\bullet)$-module~$M_\bullet$ is {\em polynomial} if
it is the restriction of a polynomial~$\GL_\bullet(\bbZ)$-module (in the sense of Definition~\ref{def:polyGL}) along the homomorphism~$\Aut(\rmN^c_\bullet)\to\Aut(\rmN^1_\bullet)=\GL_\bullet(\bbZ)$.
\end{definition}

Since a restriction of a constant module is constant, we see the constant modules are again examples of polynomial~$\Aut(\rmN^c_\bullet)$-modules.

\begin{proposition}\label{prop:Hom_is_polynomial}
For every polynomial module~$M_\bullet$, and integers~$c\geqslant1$,~$t\geqslant0$, the module
\[
\rmH_t(\Hom(\rmH_1(\rmN^c_\bullet;\bbZ),\rmH_2(\rmN^c_\bullet;\bbZ));M_\bullet)
\]
is polynomial as well.
\end{proposition}

Before we give a proof, let us refer back to Remark~\ref{rem:action} for the~$\Aut(\rmN_r^c)$-action on its module~$\Hom(\rmH_1(\rmN^c_\bullet;\bbZ),\rmH_2(\rmN^c_\bullet;\bbZ))$.

\begin{proof}
Since~$\Hom(\rmH_1(\rmN^c_\bullet;\bbZ),\rmH_2(\rmN^c_\bullet;\bbZ))$ is a free abelian group of finite rank, we have an identification
\[
\rmH_t(\Hom(\rmH_1(\rmN^c_\bullet;\bbZ),\rmH_2(\rmN^c_\bullet;\bbZ));\bbZ)\cong\Lambda^t( \Hom( \rmH_1(\rmN^c_\bullet;\bbZ) , \rmH_2(\rmN^c_\bullet;\bbZ) ) ),
\]
for the constant coefficients $\bbZ$. Since this is free abelian, the universal coefficient theorem gives
\[
\rmH_t(\Hom(\rmH_1(\rmN^c_\bullet;\bbZ),\rmH_2(\rmN^c_\bullet;\bbZ));M_\bullet)\cong\Lambda^t( \Hom( \rmH_1(\rmN^c_\bullet;\bbZ) , \rmH_2(\rmN^c_\bullet;\bbZ) ) )\otimes M_\bullet
\]
in general. It remains to be noted that~$\rmH_1(\rmN^c_\bullet;\bbZ)$ and~$\rmH_2(\rmN^c_\bullet;\bbZ)$ are polynomial. But, this follows from the Proposition~\ref{prop:homology12}: The abelianization~$\rmH_1(\rmN^{c-1}_r;\bbZ)\cong\bbZ^r$ is even linear (and independent of~$c$), whereas the Schur multiplier~$\rmH_2(\rmN^{c-1}_r;\bbZ)\cong\Lie_r^c$ is polynomial. See~\cite[I.7,~Ex.~12]{Macdonald} for instance.
\end{proof}

\begin{theorem}\label{thm:HS_nil}
For every class~$c\geqslant1$, the sequence~$\Aut(\rmN^c_\bullet)$ of automorphism groups of the free nilpotent groups of class~$c$ satisfies twisted homological stability with respect to polynomial coefficients.
\end{theorem}

\begin{proof}
We prove this by induction on~$c$. The case~$c=1$, i.e.~$\Aut(\rmN^1_\bullet)=\GL_\bullet(\bbZ)$ is dealt with by Dwyer's Theorem~\ref{thm:HS_GL}. Let us therefore assume that~$c\geqslant2$ and that the statement is true for~$c-1$. Our Theorem~\ref{thm:inflation}, applied to the extension
\[
0\longrightarrow
\Hom(\rmH_1(\rmN^{c-1}_\bullet;\bbZ),\rmH_2(\rmN^{c-1}_\bullet;\bbZ))
\longrightarrow\Aut(\rmN^c_\bullet)
\longrightarrow\Aut(\rmN^{c-1}_\bullet)
\longrightarrow 1,
\]
immediately gives the statement for~$c$. In order to be able to apply that theorem, we need to know that the modules
\[
\rmH_t(\Hom(\rmH_1(\rmN^{c-1}_\bullet;\bbZ),\rmH_2(\rmN^{c-1}_\bullet;\bbZ));M_\bullet)
\]
are polynomial for any $t\geqslant0$ and any given polynomial module~$M_\bullet$. But this follows from Proposition~\ref{prop:Hom_is_polynomial} above.
\end{proof}

%%%

\section{Other examples}\label{sec:other}

In this section, we re-derive and generalize two homological stability results involving extensions of the symmetric groups: braid groups, were homological stability for constant coefficients has originally been studied in~\cite{Arnold:braid}, and wreath products, where homological stability for constant coefficients has been proven by Hatcher and Wahl~\cite{Hatcher+Wahl}. We begin by reviewing the result that we need about the symmetric groups.

%%%

\subsection{Twisted homological stability for the symmetric groups}

Let us first review the twisted homological stability for the symmetric groups as far as we need them here. For that purpose, we let~$\Gamma^\op$ denote (a skeleton of) the category of finite pointed sets. We refer to \cite{Pirashvili1} and \cite{Pirashvili2} for the corresponding notion of a polynomial functor $F$ from~$\Gamma^\op$ to the category of abelian groups. Such a functor $F$ defines a module $M_\bullet$ with $M_r=F(\{1,\dots,r\})$, and this is then also said to be polynomial. The following result is~\cite[Theorem~4.3]{Betley02}.

\begin{theorem}\label{thm:HS_S}{\upshape\bf(Betley)}
The sequence~$\rmS_\bullet$ of symmetric groups satisfies homological stability with respect to polynomial~$\rmS_\bullet$-modules.
\end{theorem}

The proof given by Betley is an induction on the degree of the polynomial functor, based on the case of constant coefficients that is originally due to \cite{Nakaoka}. He also obtains a quantitative statement to the effect that the stable range only depends on the degree of the coefficient~$\rmS_\bullet$-module. 

%%%

\subsection{Wreath products}

Let~$G$ be any group. The wreath products~$G\wr\rmS_r=G^r\rtimes\rmS_r$ fit into a sequence~$G\wr\rmS_\bullet$ of groups. The following is a qualitative reformulation of~\cite[Theorem~1.6]{Hatcher+Wahl}.

\begin{theorem}{\upshape\bf(Hatcher, Wahl)}
For every group~$G$, the sequence~$G\wr\rmS_\bullet$ of wreath products satisfies homological stability with respect to all constant~$G\wr\rmS_\bullet$-modules.
\end{theorem}

We now give a different proof that is based upon our general result, Theorem~\ref{thm:inflation}.

\begin{proof}
First of all, we note that the statement for all constant coefficients~$A$ follows from the special case~$A=\bbZ$ and the universal coefficient theorem. And this case in turn follows from the cases in which~$A$ is one of the fields~$\bbQ$ or~$\bbZ/p$ for a prime number~$p$. 

Let~$A=\bbF$ be a field. We still need to prove that the sequence~$G\wr\rmS_\bullet$ of wreath products satisfies homological stability with respect to the constant~$G\wr\rmS_\bullet$-module~$\bbF$. By Theorem~\ref{thm:inflation}, applied to the extension
\[
1\longrightarrow
G^\bullet\longrightarrow
G\wr\rmS_\bullet\longrightarrow
\rmS_\bullet\longrightarrow
1,
\]
it suffices to check that the sequence~$\rmS_\bullet$ of symmetric groups satisfies homological stability with respect to the~$\rmS_\bullet$-modules~$\rmH_t(G^\bullet;\bbF)$ for all $t\geqslant 0$. These are no longer constant~$\rmS_\bullet$-modules, but they are polynomial~$\rmS_\bullet$-modules by the K\"unneth theorem~(for the field~$\bbF$): the tensor powers of the $\bbF$-homology of~$G$. Therefore, the result follows from Theorem~\ref{thm:HS_S}.
\end{proof}

%%%

\subsection{Braid groups}

Homological stability results for the Artin braid groups~$\rmB_r$ have first been proven in~\cite{Arnold:braid}. Other proofs have been given in other places, see again~\cite{Hatcher+Wahl}, for example. And we can also re-derive it here using our general result, Theorem~\ref{thm:inflation}.

\begin{theorem}{\upshape\bf(Arnold)}
The sequence~$\rmB_\bullet$ of Artin braid groups satisfies homological stability with respect to all constant~$\rmB_\bullet$-modules.
\end{theorem}

\begin{proof}
We will apply Theorem~\ref{thm:inflation} to the extension
\[
1
\longrightarrow\rmP_r
\longrightarrow\rmB_r
\longrightarrow\rmS_r
\longrightarrow1
\]
of the symmetric groups~$\rmS_r$ by the group~$\rmP_r$ of colored (or pure) braids. To be able to do so, it remains to observe that the~$\rmS_\bullet$-modules~$\rmH_t(\rmP_\bullet;\bbF)$ are polynomial for all~\hbox{$t\geqslant 0$}. See~\cite{Arnold:pure} for their original description, and~\cite{Brieskorn},~\cite{Vershinin1}, as well as~\cite{Vershinin2} for more recent expositions of these homology groups.
\end{proof}

The homology groups of the pure braid groups, together with the natural actions of the symmetric groups on them, form the so-called {\it Gerstenhaber operad}, as does the homology of any~$E_2$ operad, see Cohen's contribution to~\cite{Cohen}. The reader might be attracted by the idea of deriving the polynomiality of these homology groups from more general principles, and then deducing other homological stability results. This will not be pursued further here.

%%%

\section*{Acknowledgment}

This research has been supported by the Danish National Research Foundation through the Centre for Symmetry and Deformation (DNRF92). I thank Dieter Degrijse, Giovanni Gandini, and Nathalie Wahl for discussions related to the contents of this paper. 

%%%

%%%

\vfill

Markus Szymik\\
Department of Mathematical Sciences\\ 
University of Copenhagen\\
2100 Copenhagen \O\\ 
DENMARK\\
\phantom{ }\\
\href{mailto:szymik@math.ku.dk}{szymik@math.ku.dk}\\
\href{http://www.math.ku.dk/~xvd217}{www.math.ku.dk/$\sim$xvd217}

%%%

\end{document}